\documentclass[reqno,b6pape]{amsart}
\usepackage{amsmath}
\usepackage{amssymb}
\usepackage{amsthm}
\usepackage{enumerate}
\usepackage[mathscr]{eucal}
\usepackage{eqlist}
\usepackage{hyperref}
\usepackage{cite}
\hypersetup{colorlinks=true,linkcolor=red, anchorcolor=green,
citecolor=blue, urlcolor=red, filecolor=magenta, pdftoolbar=true}

\theoremstyle{plain}

\newtheorem{theorem}{Theorem}[section]
\newtheorem{lem}[theorem]{Lemma}
\newtheorem{proposition}[theorem]{Proposition}
\newtheorem{corollary}[theorem]{Corollary}
\theoremstyle{definition}
\newtheorem{definition}[theorem]{Definition}
\newtheorem{example}[theorem]{Example}

\theoremstyle{remark}
\newtheorem{remark}[theorem]{Remark}
\numberwithin{equation}{section}

\newcommand{\A}{{\mathcal{A}}}

\begin{document}
\title[A variant of d'Alembert functional equation]{A variant of d'Alembert functional equation on monoids}

\author[A. Chahbi and E. Elqorachi]{Chahbi Abdellatif and Elqorachi Elhoucien}

\address{Ibn Zohr University, Faculty of Sciences, Department of Mathematics, Agadir, Morocco}
\email{abdellatifchahbi@gmail.com-elqorachi@hotmail.com}


\begin{abstract}In this paper, we determine the complex-valued
solutions of the functional equation
$$
f(x\sigma(y))+f(\tau(y)x)=2f(x)f(y)$$ for all $x,y \in M$, where $M$
is a  monoid, $\sigma$: $M\longrightarrow M$ is an involutive
automorphism and $\tau$: $M\longrightarrow M$ is an involutive
anti-automorphism. The solutions are expressed in terms of
multiplicative functions,  and characters of $2$-dimensional
irreducible representations of $M$.
\end{abstract}
\maketitle
\section{Introduction}
It is well-known that d'Alembert's functional equation
\begin{equation}\label{EQ1}
    f(x+y)+f(x-y)=2f(x)f(y)
\end{equation} for all $x,y\in \mathbb{R}$ has its continuous
solutions $f$: $\mathbb{R}\longrightarrow \mathbb{R}$ the functions
$f(x)=0$ and $f(x)=\cosh(cx)$, where $c\in \mathbb{R}.$ Acz\'el and
Dhombres in \cite{aczel} wrote several chapters covering and
connecting the Cauchy, d'Alembert and the trigonometric functional
equations.\\The domain of definition of the solutions of (\ref{EQ1})
have progressively been extented from $\mathbb{R}$ via abelian
groups and monoids (See Acz\'el and Dhombres \cite{aczel}, Davison
\cite{8,9} and Stetk\ae r \cite{17} for details and references), and
the functional equation (\ref{EQ1}) take the general forme
\begin{equation}\label{EQ2}
    f(xy)+\mu(y)f(x\sigma(y))=2f(x)f(y)
\end{equation} for all $x,y\in S$, where $S$ is a semigroup, $\sigma$: $S\longrightarrow S$ is an involutive
anti-automorphism of $S$, i.e., $\sigma(xy)=\sigma(y)\sigma(x)$ and
$\sigma(\sigma(x))=x$ for all $x,y\in S$, and $\mu:$
$S\longrightarrow \mathbb{C}$ is a multiplicative function on $S$:
$\mu(xy)=\mu(x)\mu(y)$ for all $x,y\in S$, which satisfies
$\mu(x\sigma(x))=1$ for all $x\in S.$\\The latest result about the
functional equation (\ref{EQ2}) is due to Stetk\ae r
\cite{stetkaer}.\\ By combining an algebraic small dimension lemma
with methods by Davison \cite{8,9} on monoids, Stetk\ae r
\cite{stetkaer}  described the solution of (\ref{EQ2}) on semigroups
in terms of multiplicative functions and $2$-dimensional irreducible
representations, like they have been in previous studies of
solutions on groups and monoids \cite{8,9}.\\\\Recently, a number of
Mathematicians have studied other versions of d'Alembert's
functional equations
\begin{equation}\label{EQ3}
    f(x\sigma(y))+f(\tau(y)x)=2f(x)g(y)
\end{equation} for all $x,y\in S$ on smigroups and monoids, and
where $\sigma,\tau$: $S\longrightarrow S$ are respectively
involutives anti-automorphisms, resp. $\sigma,\tau$:
$S\longrightarrow S$ are involutives automorphisms. (See
\cite{ebanks}, \cite{redouani}, \cite{elfassi}, \cite{fadli},
\cite{ng},   \cite{perkins} and \cite{sabour}). Since in two
situations $\sigma\circ\tau$ is an automorphism, they succefully
used some steps and ideas of Stetk\ae r \cite{19} to eliminate
$\sigma$ and $\tau$ from the equation (\ref{EQ3}) and brought the
equation to a trigonometric equation,
and consequently  non-abelian phenomena like representations crop up.\\
In this paper, we solve the functional
equation\begin{equation}\label{eq23}
    f(x\sigma(y))+f(\tau(y)x)=2f(x)f(y)
\end{equation} for all $x,y\in M$, where $M$ is a monoid,  $\sigma$: $M\longrightarrow M$
is an involutive automorphism, and $\tau$: $M\longrightarrow M$ is
involutive anti-automorphism. The crucial idea used in the paper of
Stetk\ae r \cite{19} can not carry over to the our situation, in
which $\sigma\circ\tau$ is an anti-autmorphism of $M.$  The
functional equation (\ref{eq23}) is a natural generalization of the
variant of Wilson's functional equation
equation\begin{equation}\label{EQ5}
    f(xy)+f(y^{-1}x)=2f(x)g(y),\;x,y\in G\end{equation} studied by Stetk\ae r
    \cite{stetkaer1} and solved by Ebanks and Stetk\ae r \cite{stetkaer2} on
    groups.\\ Davison in \cite{8}, \cite[ Theorem 4.12]{9} and Stetk\ae r \cite[ Theorem 8.26]{17}   obtained  continuous solutions  of
the pre-d'Alembert functional equation
\begin{equation*}
g(xyz)+g(xzy)=2g(x)g(yz)+2g(y)g(xz)+2g(xy)g(z)-4g(x)g(y)g(z)
\end{equation*} for all $x,y,z \in M$ on a topological groups and monoids.\\ In  Proposition 3.2 and
Proposition 3.3 we prove some properties of the solutions of our
functional equation (\ref{eq23}). In particular, the solutions are
central and satisfies the pre-d'Alembert functional equation
(Proposition 3.3). By using this result we  give a full description
of functions $f: M \to \mathbb{C}$ satisfying the equation
(\ref{eq23}) in terms of multiplicative functions,  and characters
of $2$-dimensional irreducible representations of $M.$ Thus
functional equation (\ref{eq23})  makes sense to solve other
functional equations by expressing their solutions in terms of
solutions of (\ref{eq23}). This is what we shall do for generalized
of variant of Van Vleck's
\begin{equation}\label{eqabdo}
f(y^{-1}xz_0) - f(x\sigma{(y)}z_0) = 2f(x)f(y), \;\  x, y \in G.
\end{equation}
and for generalized of  variant Kannappan's functional equation
\begin{equation}\label{eqabdo1}
 f(x\sigma{(y)}z_0) +f(\tau(y)xz_0)= 2f(x)f(y), \;\  x, y \in M,
\end{equation} where $z_0$ is a fixed element in $M.$ Some information, applications and numerous references concerning
Van Vleck's and  Kannappan's functional equation and their further
generalizations can be found e.g.  in \cite{be}, \cite{el},
\cite{el1}, \cite{ka} and \cite{20}-\cite{va2}.
\section{\textbf{Set Up and Notation}}Throughout this paper
 $S$ designed a semigroup, $M$ a   monoid, and $G$ a  group  with neutral element  $e.$  The map  $\sigma: M \to M$ denotes an involutive
 homomorphism and the map $\tau$ an involutive anti-automorphism. That  $\sigma, \tau:  M \to M$ are involutive means that $\tau(\tau(x))=\sigma(\sigma(x)) = x$ for all $x \in M.$
 We need the following basic definitions.
\begin{definition}
\begin{enumerate}
    \item  A function $f : S \to \mathbb{C}$ is central provided that $f(xy) = f(yx)$ for all
$x, y \in S.$
\item  A function $f : S \to \mathbb{C}$ is abelian means
$$
f(x_1 x_2 \ldots x_n) = f(x_{\epsilon(1)}x_{\epsilon(2)} \ldots x
_{\epsilon(n)})
$$
for every $ n \geq 2 $  and every permutation $ \epsilon$ of   $n$
elements of $S$.

\item  A function $f : S \to \mathbb{C}$  is $\sigma$-even if $f(\sigma (x)) = f(x)$ for all $x \in S,$ and
$\sigma$-odd if $f(\sigma (x)) = -f(x)$ for all $x \in S.$
 \item A   function $\chi$: $S\longrightarrow \mathbb{C}$ is said to be multiplicative if $\chi(xy)=\chi(x)\chi(y)$ for all  $x,y\in S$.

\end{enumerate}
\end{definition}
For any    vector space $V$, we let $\mathcal{L}(V)$ denote the
algebra of linear operators of $V$ into $V$. \\A representation of
$M$ on non-zero vector space $V$ is a map $\pi : M \to \mathcal{L}(V
)$ such that $ \pi(xy) = \pi(x)\pi(y) \;\ \textrm{for all} \;\ x, y
\in M.$ The space $V$ is called the representation space of $\pi.$
 We furthermore assume that
$\pi(e) = I:\;\ \textrm{ The identity map}$.  \\Let $\dim
V<+\infty$. The dimension of the representation $\pi$ is $d_{\pi}  =
\dim V$ and the matrix coefficient  $\chi_{\pi} : M \to \mathbb{C}$
defined by $\chi_{\pi}(x) := tr(\pi(x)),\; x \in M$ is called the
character of $\pi$.\\If $M$ is a topological space, then we let
$C(M)$ denote the algebra of continuous functions from $M$ into
$\mathbb{C}.$\\
Let $V$ be a $2$-dimensional  vector space. To any linear map $A \in
\mathcal{L}(V)$ we define its adjugate ${adj}(A) \in \mathcal{L}(V)$
from linear algebra. The properties of adjugation can easily be
derived from its matrix form, and in particular we have $adj :
\mathcal{L}(V) \to \mathcal{L}(V)$ is linear,
\begin{equation}\label{eqad}
A + adj(A) = (trA)I.
\end{equation}
 Furthermore, $A
adj(A) = adj(A)A = (detA)I, \;\ adj(AB) = adj(B) adj(A), \;\
\textrm{and} \;\ adj(adj(A)) = A$ for all $A,B \in \mathcal{L}(V).$
\section{\textbf{Some properties of equation (\ref{eq23})}}
In this section, we prove some properties  of the functional
equation (\ref{eq23}) on semigroups and monoids. In particular, we
prove that solutions of of equation (\ref{eq23}) are solutions of
pre-d'Alembert functional equation.
\begin{lem}\label{lm31}
Let $f:M \to \mathbb{C}$ be a non-zero  solution of the functional
equation (\ref{eq23}). Then
\begin{equation}\label{eq31}
\begin{cases}
f(e)=1.\\
f(\sigma(x)) + f(\tau(x))=2f(x)\\
\end{cases}
\end{equation}
\end{lem}
\begin{proof}
Take $y=e$ in (\ref{eq23}) we get $f(x)=f(e)f(x).$ Since $f \neq 0$ then $f(e)=1.$\\
Putting   $ x=e$ in (\ref{eq23}) we get the following identity
\begin{equation*}
f(\sigma(y))+f(\tau(y))=2f(e)f(y).
\end{equation*}
Using that $f(e)=1$  we get
\begin{equation*}
f(\sigma(y)) + f(\tau(y))=2f(y)
\end{equation*} for all $y\in M.$ In the following proposition we
eliminate the involution anti-automorphism $\tau$ from the
functional equation (\ref{eq23}).
\end{proof}
\begin{proposition}\label{pr12}
Let $f:S \to \mathbb{C}$  be a non-zero  solution  of the functional
equation (\ref{eq23}). Then
\begin{equation}\label{eqah}
f(\sigma(a)\sigma(x)\sigma(y))+f(\sigma(a)\sigma(y)\sigma(x))=2f(\sigma(a))f(yx)+2f(x)f(\sigma(a)\sigma(y))+2f(\sigma(a)\sigma(x))f(y)-4f(\sigma(a))f(x)f(y),
\end{equation}
for all  $ a,x,y \in S.$
\end{proposition}
\begin{proof} For the pair
$(\sigma(a)\sigma(x), y)$ the functional equation (\ref{eq23} ) can
be written as follows
\begin{equation}\label{eq32}
f(\sigma(a)\sigma(x)\sigma(y)) +
f(\tau(y)\sigma(a)\sigma(x))=2f(\sigma(a)\sigma(x))f(y).
\end{equation}
By applying (\ref{eq23} ) to the pair $(\tau(y)\sigma(a),x)$  we
obtain
\begin{equation}\label{eq33}
f(\tau(y)\sigma(a)\sigma(x)) +
f(\tau(x)\tau(y)\sigma(a))=2f(\tau(y)\sigma(a))f(x)=2f(x)[2f(\sigma(a))f(y)-f(\sigma(a)\sigma(y))].
\end{equation} By using (\ref{eq23}),
we reformulate the second term on the left hand side of (\ref{eq33})
as follows
\begin{equation*}
f(\tau(x)\tau(y)\sigma(a))=f(\tau(yx)\sigma(a))=2f(\sigma(a))f(yx)-f(\sigma(a)\sigma(yx)),
\end{equation*}
which turns the identity (\ref{eq33}) into  we obtain
\begin{equation*}
f(\tau(y)\sigma(a)\sigma(x))
+2f(\sigma(a))f(yx)-f(\sigma(a)\sigma(y)\sigma(x))=4f(\sigma(a))f(x)f(y)-2f(x)f(\sigma(a)\sigma(y)).
\end{equation*}
Subtracting this from (\ref{eq32}) we get after some simplifications
that
\begin{equation}\label{eq34}
f(\sigma(a)\sigma(x)\sigma(y))+f(\sigma(a)\sigma(y)\sigma(x))=2f(\sigma(a))f(yx)+2f(x)f(\sigma(a)\sigma(y))+2f(\sigma(a)\sigma(x))f(y)-4f(\sigma(a))f(x)f(y)
\end{equation} for all $a,x,y\in S.$ This ends the proof.
\end{proof}

\begin{proposition}\label{pr22}
Let $f:S \to \mathbb{C}$ be a non-zero  solution  of the functional
equation (\ref{eq23}). Then
\begin{enumerate}
    \item  $f$ is  central.
    \item $
f(\sigma(x)) + f(\tau(x))=\lambda f(x)$ for all $x \in S$, and for
same $\lambda \in \mathbb{C}^{*}.$
\item $f$ is a solution of the pre-d'Alembert functional equation
i.e.,
\begin{equation*}
f(axy)+f(ayx)=2f(a)f(xy)+2f(x)f(ay)+2f(ax)f(y)-4f(a)f(x)f(y)
\end{equation*}
for all  $ a,x,y \in S.$

\end{enumerate}

\end{proposition}
\begin{proof}
\begin{enumerate}
    \item According to Proposition \ref{pr12}  $f$ satisfies (\ref{eqah}). Interchange of $x$ and $y$ in (\ref{eqah}) and comparing the result obtained with (\ref{eqah})  we  find  that
 $f(\sigma(a))f(xy)=f(\sigma(a))f(yx)$ for all  $ a,x,y \in S.$  Since $f \neq 0,$ then $f$ is central.
    \item Replace $x$ by $\sigma(x)$ in (\ref{eq23}), we obtain
    \begin{equation}\label{eqel}
    f(\sigma(x)\sigma(y))+f(\tau(y) \sigma(x))=2f(\sigma(x))f(y), \;\ x,y \in S.
    \end{equation}
    Substitute $x$ by $\tau(x)$ in in (\ref{eq23}), we obtain
    \begin{equation}\label{eqel1}
    f(\tau(x)\sigma(y))+f(\tau(y) \tau(x))=2f(\tau(x))f(y), \;\ x,y \in S.
    \end{equation}
    Adding  (\ref{eqel}) and  (\ref{eqel1}) and use that $f$ is central we obtain
    $$[f(\sigma(y)\sigma(x))+f(\tau(x)\sigma(y))]+[f(\tau(y)\sigma(x))+f(\tau(x)\tau(y))]$$$$=
    f(x)(f(\sigma(y))+f(\tau(y)))=f(y)(f(\sigma(x))+f(\tau(x)))
    $$ for all $x,y \in S$. Since $f \neq 0$ then there exists $y_0 \in S$ such that $f(y_0) \neq 0,$
and  that
    $f(\sigma(x))+f(\tau(x))= \lambda f(x)$ for all $x \in S$, and  where $\lambda =
    \frac{f(\sigma(y_0))+f(\tau(y_0))}{f(y_0)}.$\\
    Now, if $\lambda =0,$ we get $f(\sigma(x))=-f(\tau(x))$ for all $x \in S.$ By replacing $x$ by $\sigma(x)$, $y$ by $\tau(y)$ respectively
     $x$ by $\tau(x)$ and $y$ by $\sigma(y)$ in (\ref{eq23}) we
    get
        $$
        \begin{cases}
        f(\sigma(x)\sigma(\tau(y)))+f(y\sigma(x))=2f(\sigma(x))f(\tau(y)) \\
        f(\tau(x)y)+f(\tau(\sigma(y))\tau(x))=2f(\tau(x))f(\sigma(y))
        \end{cases}
        $$
        Since $f(\sigma(x))=-f(\tau(x)), \;\ x \in S$ and $f$ is central,
        so by adding the two functional equations of the above system
        we get
        $$
        -f(\tau(x)y)-f( y\sigma(x))+f(y\sigma(x))+f(\tau(x)y)= -4 f(\sigma(x))f(\sigma(y))=0
        $$ for all $x,y \in S$. Since $\sigma^{2}=I,$ so $f=0,$ which  is a contradiction  with $ f
\neq 0.$
    \item Let $a, x, y \in S$ be arbitrary.
For the pair $(\tau(x)\tau(a), y)$  equation (\ref{eq23} implies
that
\begin{equation}\label{eq321}f(\tau(x)\tau(a)\sigma(y))+
f(\tau(y)\tau(x)\tau(a)) =2f(\tau(x)\tau(a))f(y).
\end{equation}
By applying (\ref{eq23} ) to the pair $(\tau(a)\sigma(y),x)$  we
obtain
\begin{equation}\label{eq331}f(\tau(a)\sigma(y)\sigma(x))+
f(\tau(x)\tau(a)\sigma(y))
=2f(\tau(a)\sigma(y))f(x)=2f(x)[2f(\tau(a))f(y)-f(\tau(y)\tau(a))].
\end{equation}
By using (\ref{eq23} ) we reformulate the first term on the left
hand side of (\ref{eq331}) as follows
\begin{equation*}
f(\tau(a)\sigma(yx))=2f(\tau(a))f(yx)-f(\tau(x)\tau(y)\tau(a)),
\end{equation*}
which turns the identity (\ref{eq331}) into
\begin{equation*}
2f(\tau(a))f(yx)-f(\tau(x)\tau(y)\tau(a))+f(\tau(x)\tau(a)\sigma(y))=4f(\tau(a))f(x)f(y)-2f(x)f(\tau(y)\tau(a)).
\end{equation*}
Subtracting this from (\ref{eq321}) we get after some
simplifications that
\begin{equation}\label{eq341}
f(\tau(axy))+f(\tau(ayx))=2f(\tau(a))f(yx)+2f(x)f(\tau(ay))+2f(\tau(ax))f(y)-4f(\tau(a))f(x)f(y)
\end{equation}
By adding  (\ref{eqah}) and (\ref{eq341}), and using Proposition 3.3
(2), we conclude that $f$ is a solution of pre-d'Alembert functional
equation on semigroup $S.$ This completes the proof.
\end{enumerate}
\end{proof}
Now, we are ready to prove the main result.
\begin{theorem}\label{th31} \textbf{Abelian solutions of equation
(\ref{eq23})}.\\ The abelian solutions $f:M\rightarrow \mathbb{C}$
of (\ref{eq23}) are the functions of the form
$$f=\frac{\chi+\chi\circ\sigma\circ\tau}{2},$$ where $\chi:M\rightarrow \mathbb{C}$ is a multiplicative function such that:\\
\textbf{(i)} $\chi\circ\sigma\circ\tau=\chi\circ\tau\circ\sigma,$ and\\
  {\textbf{(ii)}} $\chi$ is $\sigma$-even and/or $\tau$-even.
\end{theorem}

\begin{proof}
It is elementary to check that the functions stated in the Theorem
\ref{th31} define abelian solutions, so it is left to show that any
abelian solution $f: M \to  \mathbb{C}$ of (\ref{eq23}) can be
written as in the form in Theorem \ref{th31}. We have from
proposition \ref{pr22} that $f$ is  solution of pre-d'Alembert
functional equation on a Monoid $M,$ and which satisfies $f(e)=1$,
(see Lemma 3.1). If $f$ is an abelian solution by \cite[Theorem
8.13.]{17} then there exist multiplicative functions $\chi_1,\chi_2
: M \to \mathbb{C}$ such that $ f=\frac{\chi_1+\chi_2}{2}.$ We will
discuss two cases.\\ Case 1. If $\chi_1=\chi_2,$ then letting
$\chi:=\chi_1=\chi_2$ we have $f=\chi.$ Substituting $f=\chi$ into
(\ref{eq23}) we get that $\chi\circ\sigma+\chi\circ\tau=2\chi.$ So
$\chi=\chi\circ\sigma=\chi\circ\tau.$ Then $f$ has the desired
form.\\Case 2. If $\chi_1\neq\chi_2,$ substituting
$f=(\chi_1+\chi_2)/2$ into (\ref{eq23}) we find after a reduction
that
\begin{eqnarray*}
&&\chi_1(x)[\chi_1(\sigma(y))+\chi_1(\tau(y))-\chi_1 (y)-\chi_2(y)] \\
&&+\chi_2(x)[\chi_2(\sigma(y))+\chi_2(\tau(y))-\chi_1
(y)-\chi_2(y)]]=0
\end{eqnarray*}
for all $x,y\in M.$ Since $\chi_1\neq\chi_2,$ we get from the theory
of multiplicative functions (see for instance (\cite[Theorem
3.18]{17}) that both terms are $0,$ so
\begin{equation}\label{eq348}
\left\{
    \begin{aligned}
    &\chi_1(x)[\chi_1(\sigma(y))+\chi_1(\tau(y))-\chi_1 (y)-\chi_2(y)]=0&\\
    &\chi_2(x)[\chi_2(\sigma(y))+\chi_2(\tau(y))-\chi_1 (y)-\chi_2(y)]=0&
    \end{aligned}
    \right.
\end{equation}
for all $x,y\in M.$ Since $\chi_1 \neq\chi_2$ at least one of $\chi_1$ and $\chi_2$ is not zero.\\
Now if $\chi_1 \neq 0$ and $ \chi_2=0.$ It is allowed that
$\chi_1(\sigma(y))+\chi_1(\tau(y))=\chi_1 (y)$ this implies that
$\chi_1(\sigma(y))=0 \;\ \textrm{or} \;\  \chi_1(\tau(y))=0,$ in
either case  $\chi_1 = 0$ because $\sigma^2= \tau^2=id.$

We have now   $ \chi_1 \neq 0$ and $\chi_2 \neq 0.$ From
(\ref{eq348}), we get
\begin{equation*}
\chi_1 +\chi_2=\chi_1 \sigma+\chi_1 \circ \tau=\chi_2 \circ
\sigma+\chi_2 \circ \tau .
\end{equation*}
Using that $\chi_1\neq\chi_2,$ we obtain that $\chi_2 =\chi_1 \circ
\sigma\circ\tau=\chi_1 \circ \tau\circ\sigma $ now as we see that
$\chi_1 +\chi_2=\chi_1 \sigma+\chi_1 \circ \tau $ this implies that
$\chi_1 $ is is $\sigma$-even or $\tau$-even.  Finally,  we deduce
that $f$ has the form  stated in Theorem \ref{th31} with
$\chi=\chi_1.$\end{proof}
\begin{theorem}\label{th311}  \textbf{Non-abelian solutions of equation
(\ref{eq23})}.\\ The non-abelian solutions $f$ $M\longrightarrow
\mathbb{C}$ of (\ref{eq23}) are of the form
$f(x)=\frac{1}{2}tr\pi(x)$ for all $x\in M$, where $\pi$ is an
irreducible, $2$-dimensional representation of $M$ for which
$adj(\pi \circ \sigma)=\pi\circ\tau,$ $
\textrm{tr}(\pi(x))=\textrm{tr}(\pi(\sigma(x)))=\textrm{tr}(\pi(\tau(x)))=\textrm{tr}(\pi(\sigma\circ\tau(x)))=\textrm{tr}(\pi(\tau\circ\sigma(x))),$
 for all $ x \in M.$
\\If $\pi'$ is any irreducible representation of $M$ on a finite dimensional vector space, such that $f(x)=\frac{1}{2}tr\pi'(x)$ for all $x\in M$,
then $\pi'$ and $\pi$ are  equivalent.
\end{theorem}\begin{proof} By Proposition 3.3 (3),
 $f$ is a non-abelian   solution of  the pre-d'Alembert functional
 equation on $M.$ Then from \cite[ Theorem 8.26]{17} $f$ has the form
 $f=\frac{1}{2}tr(\pi)$, where $\pi$ is an irreducible,  $2$-dimensional
 representation of $M$. Substituting $f$ into the functional equation (\ref{eq23}) and using that $f$ is central we get $$ tr\biggl(
\pi(x)(\pi(\sigma(y)+\pi(\tau(y))-tr(\pi(y))I) ) \biggr)=0, \;\ x,y
\in M.$$ Now, Burnside Theorem's \cite{lomonosov} shows that
$$ tr\biggl(
A(\pi(\sigma(y)+\pi(\tau(y))-tr(\pi(y))I ) \biggr)=0,$$ for all
$A\in \mathbf{L}(V)$ and for all $y \in M,$ and consequently, we get
\begin{equation}\label{bur}
\pi(\sigma(y)+\pi(\tau(y))=tr(\pi(y))I_2\end{equation} for all $y
\in M.$ Replacing $y$ by $\sigma(y)$ respectively by $\tau(y)$ in
(\ref{bur}) we obtain
\begin{equation} \label{eqabdel}
\begin{cases}
\pi(y)+\pi(\tau(\sigma(y)))=tr(\pi(\sigma(y)))I_2\\
\pi(y)+\pi(\sigma(\tau(y)))=tr(\pi(\tau(y)))I_2
\end{cases}
\end{equation}
 By adding the two functional equation of the above system and
using that $f(\sigma(y))+f(\tau(y))=2f(y)$ for all $y\in M,$ we
obtain
$$
\pi(y)+\frac{\pi(\tau(\sigma(y)))+\pi(\sigma(\tau(y)))}{2}=tr(\pi(y))I_2,
$$ for all $y\in M.$ From the properties of adjugation (2.1) we get
$adj(\pi(y))=\frac{\pi(\tau(\sigma(y)))+\pi(\sigma(\tau(y)))}{2},$
for all $y\in M.$ Since $adj(\pi(xy))=adj(\pi(x)\pi(y))=adj(\pi(y))
adj(\pi(x))$ then \begin{equation} \label{eqabdel1}
\pi(\tau(\sigma(y)))(\pi(\tau(\sigma(x)))-\pi(\sigma(\tau(x))))=\pi(\sigma(\tau(y)))(\pi(\tau(\sigma(x)))-\pi(\sigma(\tau(x)))),
\;\ x,y \in M.
\end{equation}
Now by using (\ref{eqabdel}) we get
$$(\pi(\tau(\sigma(y)))-\pi(\sigma(\tau(y))))=(tr(\pi(\sigma(y)))-tr(\pi(\tau(y))))I_2,
\;\ y \in M,$$ since (\ref{eqabdel1}) then
$adj(\pi(y))=\pi(\tau(\sigma(y)))=\pi(\sigma(\tau(y))),$ this
implies that $adj(\pi(\sigma(y)))=\pi(\tau(y)).$ By the condition
$f(\sigma(y))+f(\tau(y))=2f(y)$ we see that
$tr(\pi(\sigma(y)))+tr(\pi(\tau(y)))=2 tr(\pi(y),$ so we get
$tr(\pi(\sigma(y)))= tr(\pi(y).$ Finally we obtain the desired
non-abelian solution.
\end{proof}
\begin{remark}
By using similar proof  (see Proposition 3.3), we show the
centrality of the solutions of the functional equation (1.4), when
$\sigma$ is an involutive anti-automorphism and $\tau$ is in
involutive automorphism. So, the order of $\sigma$ and $\tau$ in the
functional equation (1.4) is not important.
\end{remark}
The following  corollaries are immediate consequences of Theorem
\ref{th31} and Theorem \ref{th311}.

\begin{corollary} Let $G$ be a  topological group  and $\sigma$: $G\longrightarrow G$ be a continuous involutive automorphism of $G$. There are two non-zero solutions $f:G\rightarrow \mathbb{C}$ of the functional equation
$$
f(x\sigma(y))+f(y^{-1}x)=2f(x)f(y), \ x,y \in G,$$ abelian and
non-abelian:

\begin{enumerate}
    \item [$\rm(1)$] Every {abelian} continuous solution $f$ is of the form
   $$f(x)=\frac{\chi(x)+\chi(\sigma(x^{-1})}{2},$$ where $\chi:G\rightarrow \mathbb{C}$ is a continuous multiplicative function such that $\chi$ is $\sigma$-even and/or $\chi(x)=\chi(x^{-1})$.

\item[$\rm(2)$] Every {non-abelian} continuous solution $f$ is of the form
$$f(x)=\frac{1}{2}tr\pi(x), \ \ \ x\in G,$$ where $\pi$ is a continuous irreducible unitary representation of dimension $2$ for which $adj(\pi(\sigma(x)))=\pi(x^{-1}),$ and  $tr(\pi(\sigma(x)))=tr(\pi(x))$.
\end{enumerate}

\end{corollary}
\begin{proof}
It suffices to take $\tau(x) = x^{-1}$ for all $x \in G$ in Theorem
\ref{th31} and Theorem \ref{th311}.
\end{proof}

\begin{corollary} Let $M$ be a topological monoid and $\tau$ be a continuous involutive anti-automorphism of $M$. There are two non-zero continuous solutions $f:M\rightarrow \mathbb{C}$ of the functional equation
$$
f(xy)+f(\tau(y)x)=2f(x)f(y), \ x,y \in G,$$ abelian and non-abelian:

\begin{enumerate}
    \item [$\rm(1)$] Every {abelian} continuous solution $f$ is of the form
   $$f=\frac{\chi+\chi\circ\tau}{2},$$ where $\chi:M\rightarrow \mathbb{C}$ is a continuous multiplicative function.

\item[$\rm(2)$] Every {non-abelian} continuous solution $f$ is of the form
$$f(x)=\frac{1}{2}tr\pi(x), \ \ \ x\in M,$$ where $\pi$ is a continuous irreducible unitary representation of dimension $2$ for which $adj(\pi)=\pi\circ\tau.$
\end{enumerate}
\end{corollary}
\begin{proof}
It suffices to take $\sigma(x) = x$ for all $x \in M$ in Theorem
\ref{th31} and Theorem \ref{th311}.
\end{proof}

\begin{corollary}Let $G$ be a topological group. There are two non-zero continuous solutions $f:G\rightarrow \mathbb{C}$ of the functional equation
$$
f(xy)+f(y^{-1}x)=2f(x)f(y), \ x,y \in G,$$ abelian and non-abelian:

\begin{enumerate}
    \item [$\rm(1)$] Every {abelian} continuous solution $f$ is of the form
   $$f(x)=\frac{\chi(x)+\chi(x^{-1})}{2}, \;\ x \in G$$ where $\chi:G\rightarrow \mathbb{C}^{*}$ is a continuous multiplicative function.

\item[$\rm(2)$] Every {non-abelian} continuous solution $f$ is of the form
$$f(x)=\frac{1}{2}tr\pi(x), \ \ \ x\in G,$$ where $\pi$ is a continuous irreducible unitary representation of dimension $2$ for which $adj(\pi(x))=\pi(x^{-1}).$
\end{enumerate}
\end{corollary}
\begin{proof}
It suffices to take $\sigma(x) = x$ and $\tau(x)=x^{-1}$ for all $x
\in G$ in Theorem \ref{th31} and Theorem \ref{th311}.
\end{proof}
\section{Some applications and examples}
In this section, we determine solutions of the functional equations
(\ref{eqabdo}) and (\ref{eqabdo1}). We show that the solutions are
closely related to the solutions  of (\ref{eq23}).
\begin{proposition}\label{prabdo}
Let  $z_0$ in $ Z(G)$ the center of $G$ be given. Let $f : G \to
\mathbb{C}$ be a non-zero solution of the functional equation
(\ref{eqabdo}). Then
\begin{enumerate}
    \item $f(e)=0.$
    \item $f(y^{-1}z_0)=f(\sigma(y)z_0), \;\ y \in G.$
    \item $f(y^{-1})=-f(\sigma(y)), \;\ y \in G.$
    \item $f(y \sigma(z_0)z_0)=-f(y), \;\ y \in G.$
    \item $f(\sigma(z_0) z_0)=0.$
    \item $ f(z_0)=1.$
    \item $ f(z_0^{2})=0.$
\end{enumerate}
\end{proposition}
\begin{proof}
\begin{enumerate}
    \item Substituting $y=e$ into (\ref{eqabdo}) and using that $f \neq 0,$ we get that $f(e)=0.$
        \item Take $x=e$ in (\ref{eqabdo}) and using that $f(e)=0$ we obtain $f(y^{-1}z_0)=f(\sigma(y)z_0)$ for all $y \in G.$
        \item Replacing $x$ by $\sigma(x^{-1})$ in (\ref{eqabdo}) we get
        $$
f(y^{-1}\sigma(x^{-1})z_0) - f(\sigma(x^{-1})\sigma(y)z_0) =
2f(\sigma(x^{-1}))f(y), \;\  x, y \in G.
$$
 This implies that      $$
f((\sigma(x)y)^{-1}z_0) - f((\sigma(y^{-1})\sigma(x))^{-1}z_0) =
2f(\sigma(x^{-1}))f(y), \;\  x, y \in G.
$$
Using (2) we get \begin{equation}\label{eqabdo2}
f(x\sigma(y)z_0)-f(y^{-1}xz_0)   = 2f(\sigma(x^{-1}))f(y), \;\  x, y
\in G.
\end{equation}
From (\ref{eqabdo})  and  that $f \neq 0$ we obtain
$f(x^{-1})=-f(\sigma(x))$ for all $x\in G.$
\item Taking $y=\sigma(x) z_0$ into (2) and using that $z_0 \in Z(G)$ we get $f(\sigma(x^{-1}))=f(x \sigma(z_0) z_0)$
since $f(x^{-1})=-f(\sigma(x)),$ then $f(x \sigma(z_0) z_0)=-f(x).$
\item Putting $y=z_0$ in (2) we get $ 0=f(e)=f(\sigma(z_0) z_0).$
\item Replacing $y$ by $z_0$ in (\ref{eqabdo}) and using $z_0 \in Z(G)$  we have $f(x)- f(x \sigma(z_0) z_0)=2f(x)f(z_0),$
since  $f(x \sigma(z_0) z_0)=-f(x),$ so $2f(x)=2f(x)f(z_0),$ but $ f
\neq 0$ then $ f(z_0)=1.$
\item Replacing $y$ by $z_{0}^{-1}$ in (2) and using (3) we obtain
$f(z_{0}^{2})=f(\sigma(z_{0}^{-1})z_0)=f(\sigma(z_{0}^{-1}\sigma(z_0))=-f(\sigma(z_{0})^{-1}z_0)$
then $f(z_{0}^{2})=0$
\end{enumerate}
\end{proof}
\begin{lem}\label{lemabdo}
Let  $z_0$ in $ Z(G)$ be given. Let $f : G \to \mathbb{C}$ be a
non-zero solution of the functional equation (\ref{eqabdo}). Then
\begin{enumerate}
    \item $f(z_{0}^{-1})^{2}=f(\sigma(z_0))^2=1.$
    \item $ f(z_{0}^3)=f(z_{0}^{-1})=-f(\sigma(z_0)).$
\end{enumerate}
\end{lem}
\begin{proof}
\begin{enumerate}
    \item If we take $x=y=z_{0}^{-1}$ in (\ref{eqabdo}) and using
    $z_0 \in Z(G)$ we get
    $$
    f(z_0)-f(\sigma((z_{0})^{-1}))=2f((z_{0})^{-1})^2
    $$
Using that $f(z_{0})=1$ and $f(y^{-1})=-f(\sigma(y)), \;\ y \in G,$
we get $f(z_{0}^{-1})^{2}=f(\sigma(z_0))^2=1.$
\item Replacing  $y$ by $((z_0)^{2})^{-1}$ in Proposition 4.3 (2)  and using and Proposition 4.3 (4) we get
$$
    f(z_{0}^3)=f(\sigma(z_0)^2z_{0})=f(\sigma(z_0)\sigma(z_0)z_0)=-f(\sigma(z_0)).
$$
\end{enumerate} This ends the proof.
\end{proof}
\begin{theorem}
Let $G$ be a topological group  and let $z_0 \in Z(G)$ be given. Let
$f : G \to \mathbb{C}$ be a continuous solution of the functional
equation (\ref{eqabdo}).
\begin{enumerate}
    \item If $\sigma(z_0) z_0 =e$, then  (\ref{eqabdo}) has no non-zero solutions.
    \item $\sigma(z_0) z_0 \neq e,$ then the solutions have the following
    forms

    \begin{enumerate}
        \item $ f(x)= \frac{ \chi(x) - \chi(x^{-1})}{2 \chi(z_0)}, \;\  x \in G $, where $ \chi: G \to \mathbb{C}^{*}$
         is a continuous multiplicative function such that $\chi(\sigma(x))=\chi(x)$ and  $\chi( z_0)^{2}=-1.$
        \item $f(x)= \frac{ \chi(x) - \chi(\sigma(x))}{2 \chi(z_0)},$ where $ \chi: G \to \mathbb{C}^{*}$ is a  continuous multiplicative function such that
        $\chi(x) =\chi(x^{-1})$ and $ \chi( \sigma(z_0)z_0)=-1.$
    \end{enumerate}
\end{enumerate}
\end{theorem}
\begin{proof}
\begin{enumerate}
    \item If $\sigma(z_0) z_0 =e,$ by Proposition \ref{prabdo} (4) $f(x \sigma(z_0)z_0)=-f(x), \;\ x \in G,$ so    we have $f(x)=0$ for all $x \in G.$
        \item   From now,  we assume that $\sigma(z_0) z_0 \neq e.$ Replacing $x$ by $x z_0$ and $y$ by $y z_0$ in (\ref{eqabdo}) and using
        that $ z_0 \in Z(G)$  we get $$ f(y^{-1} x z_0) - f(x \sigma(y)z_0 \sigma(z_0) z_0)=2f(x z_0)f(y z_0)$$
        since $f(x \sigma(z_0)z_0)=-f(x), \;\ y \in G,$ then  $$ f(y^{-1} x z_0) +f(x \sigma(y) z_0 )=2f(x z_0)f(y z_0) , \;\ x, y \in G.$$
    Consider $g(x)=f(x z_0), \;\ x \in G,$ then the function $g$ satisfies (\ref{eq23})  with $\tau(x)=x^{-1}$ and $
    g(e)=1,$ so $g\neq0$. As $g$ is a solution of equation (\ref{eq23})  then by Proposition
\ref{pr22} (3) $g$ is a solution of pre-d'Alembert functional
equation. From Proposition 4.1 (7) we get $g(z_0)=0$ and by Lemma
\ref{lemabdo} we have $d(z_0) = 2g(z_0)^2 - g(z_{0}^2)=0-f(z_{0}^3)=
0 - (-f(\sigma(z_0)) = f(\sigma(z_0)) \neq 0.$  So, we get $g(z_0)^2
\neq d(z_0).$ According to \cite[ Proposition 8.14(a)]{17} we have
$g$ is abelian. Furthermore, from Corollary 3.6 (1) $g
(x)=\frac{\chi(x)+\chi(\sigma(x^{-1}))}{2}$ for all $ x \in G,$
where $\chi$ is a non-zero continuous multiplicative function  such
that $\chi$ is $\sigma$-even and/or $\chi(x) = \chi(x^{-1})$. Since
$g(z_0)=0,$ so $ \chi(z_0)=-\chi(\sigma(z_{0}^{-1}))$, which implies
that  $\chi(z_0\sigma(z_{0}))=-1.$  We have $f(x)=g(x z_{0}^{-1})$
and $\chi(z_0\sigma(z_{0}))=-1$, so we get that $ f(x)= \frac{
\chi(x) - \chi(\sigma(x^{-1}))}{2 \chi(z_0)}.$ This completes the
proof.
\end{enumerate}
\end{proof}
\begin{theorem}\label{th3122} Let $M$ be a topological monoid.
Let $z_{0}$ be an arbitrarily fixed element in $M.$  There are two
non-zero continuous solutions $f:G\rightarrow \mathbb{C}$ of
(\ref{eqabdo1}), abelian and non-abelian.
\begin{enumerate}
    \item [$\rm(1)$]  Every {abelian} continuous solution $f$ is of the form
   $$f=f(e)\frac{\chi+\chi\circ\sigma\circ\tau}{2},$$ where $\chi:G\rightarrow \mathbb{C}$ is a continuous  multiplicative function such that:
\begin{enumerate}
  \item [$\rm(i)$] $\chi\circ\sigma\circ\tau=\chi\circ\tau\circ\sigma,$
  \item[$\rm(ii)$] $\chi$ is $\sigma$-even and/or $\tau$-even, and
    \item [$\rm(iii)$] $\chi(\sigma(\tau(z_0)))=\chi(z_0)=f(e).$
\end{enumerate}
\item[$\rm(2)$]   Every {non-abelian} continous solution $f$ is of the form
$$f(x)=\frac{f(e)\chi_{\pi}(x)}{2}, \ \ \ x\in G$$ where $\pi$ is a continuous irreducible  representation of dimension $2$
  for which  $adj(\pi \circ \sigma)=\pi\circ\tau,$  $ \textrm{tr}(\pi(x))=\textrm{tr}(\pi(\sigma(x))), \;\ x \in M,$ and $\pi(z_{0})= f(e) I_2$
\end{enumerate}
\end{theorem}
\begin{proof}
Taking $y = e$ in (\ref{eqabdo1}), we get that
\begin{equation}\label{eqaaa}
f(xz_{0}) = f(x)f(e), x \in M.
\end{equation}
We have  $f \neq 0,$ so  $f(e) \neq 0.$ Using (\ref{eqabdo1}) and
(\ref{eqaaa}) twe get that the function $
g(x)=\frac{f(x)}{f(z_{0})}$ is a solution of  (\ref{eq23}).
Therefore, by Theorem \ref{th31}, Theorem \ref{th311} and the identity (\ref{eqaaa}) it is easy to   get the form of solutions of (\ref{eqabdo1}).\\
Conversely, It is easy to check that the all of the possibilities
listed in Theorem \ref{th3122} define solutions of (\ref{eqabdo1}).
\end{proof}
\begin{example}
For a non-abelian example of Monoid, we consider $G=M(2,\mathbb{C})$
and take $\sigma$ as an automorphism and $\tau$ as an
anti-automorphism such that
\begin{equation*}
\sigma\left(
        \begin{array}{cc}
          a & b \\
          c & d \\
        \end{array}
      \right)=J\left(
                 \begin{array}{cc}
                   a & b \\
                   c & d \\
                 \end{array}
               \right) J=\left(
                 \begin{array}{cc}
                   d & c \\
                   b & a \\
                 \end{array}
               \right)
\end{equation*}
where $J=\left(
                 \begin{array}{cc}
                   0 & 1 \\
                   1& 0 \\
                 \end{array}
               \right)$
 and
\begin{equation*}
\tau\left(
        \begin{array}{cc}
          a & b \\
          c & d \\
        \end{array}
      \right)=\left(
                 \begin{array}{cc}
                   a & -b \\
                   -c & d \\
                 \end{array}
               \right)
\end{equation*}
We indicate here the corresponding continuous solutions of
(\ref{eq23}). Simple computations show that
\begin{equation*}
\sigma \circ \tau =\tau\circ\sigma
\end{equation*}
Note that $$T=\left(
                   \begin{array}{cc}
                     a & c \\
                     b & d \\
                   \end{array}
                 \right).$$
    \textbf{Abelian solutions}.\\
The continuous non-zero multiplicative functions on $M(2, C)$ are
given   (see \cite[ Example 5.6]{2}):  $\mu= 1,$ or else
$$
\begin{cases}
 \chi(T)=\left|\det(T)\right|^{\lambda-n}(det(T))^{n} &\textrm{ when} \;\ det(T) \neq 0\\
0                                                    & \textrm{when}
\;\ det(T) = 0,
 \end{cases}
$$
where $\lambda \in \mathbb{C}$  with $\textrm{Re}{\lambda} > 0$ and
$n \in  \mathbb{Z}.$
It is clear that $ \chi(\sigma(\tau(T)))= \chi(\tau(\sigma(T))),$ $\chi(\sigma(T))= \chi(\tau(T))=\chi(T).$ It follows that the abelian  continuous solutions  are $f(T)=\left|ad-bc\right|^{\lambda-n}(ad-bc)^{n}.$\\
\textbf{Non-abelian solutions}.\\
It is clear that $tr(T)=tr(\sigma(T))$, $
\textrm{Adj}(\sigma(T))=\tau(T)$ and  $\sigma\circ\tau
(T)=\tau\circ\sigma (T).$ By using Theorem \ref{th31}, the
non-abelian non-zero continuous solutions $f:M(2,
\mathbb{C})\rightarrow \mathbb{C}$ of (\ref{eq32}) are
$f=\frac{1}{2}tr.$

\end{example}

\begin{example}

        The $3$-dimension  Heisenberg group $G = H_3$ described in \cite[ Example A.17(a)]{17},
   and take as the involutive automorphism (see \cite{2}) $$\sigma\begin{pmatrix}
     1 & a & c  \\
      0 & 1 & b\\
0 & 0 & 1
        \end{pmatrix}=\begin{pmatrix}
     1 & -b & -c+ab  \\
      0 & 1 & -a\\
0 & 0 & 1
        \end{pmatrix}$$     and $$\sigma\begin{pmatrix}
     1 & a & c  \\
      0 & 1 & b\\
0 & 0 & 1
        \end{pmatrix}=\begin{pmatrix}
     1 & -a & -c+ab  \\
      0 & 1 & -b\\
0 & 0 & 1
        \end{pmatrix}.$$
        It easy to see that $$\sigma \circ \tau = \tau \circ \sigma\begin{pmatrix}
     1 & a & c  \\
      0 & 1 & b\\
0 & 0 & 1
        \end{pmatrix}=\begin{pmatrix}
     1 & b & c  \\
      0 & 1 & a\\
0 & 0 & 1
        \end{pmatrix}.$$
    The continuous non-zero multiplicative functions on $G$ are given   (see [\cite{17}):
$$
\mu_{\alpha,\beta}(a,b)=\exp^{\alpha a + \beta b}, \;\ \alpha, \beta
\in \mathbb{C}.
$$
 We compute that $\mu=\mu \circ \sigma $ if and only if $\beta =-\alpha,$ and that $\mu=\mu \circ \tau $ if and only if $\alpha=\beta =0.$
Then
$$
f((a,b))= (\exp^{\alpha (a -  b)} + \exp^{\alpha (b - a)} )/2
$$

 By \cite[Theorem 8.29.]{17}  Any continuous solution of the pre-d'Alembert functional equation  on
$G,$ is abelian.  In particular it has no  non-abelian solution on
$G.$

\end{example}

\end{document}